%% file: act-arxiv-cleaned.tex
\RequirePackage[l2tabu, orthodox]{nag}
\RequirePackage{snapshot}

\documentclass[10pt,onecolumn]{article}

\usepackage[dvips,letterpaper,top=0.5in, bottom=0.5in, left=0.75in, right=0.5in,includefoot,heightrounded]{geometry}

\usepackage{srcltx}

\usepackage{amsmath}
\usepackage{amssymb,amsfonts}

\usepackage{abstract}

\usepackage{epsfig}
\usepackage[usenames,dvipsnames]{color}
\usepackage[usenames,dvipsnames]{xcolor}
\usepackage{subfigure}

\usepackage{booktabs}

\usepackage{setspace}
\usepackage{flushend}
\usepackage{multicol}

\usepackage{cite}
\usepackage{url}
\usepackage[normalem]{ulem}

\usepackage{enumerate}

\usepackage{algorithmic}

\usepackage{hyperref}

\usepackage[sc,tiny,center]{titlesec}
       \usepackage{mathptmx}

\newtheorem{proposition}{Proposition}
\newtheorem{definition}{Definition}
\newtheorem{theorem}{Theorem}
\newtheorem{lemma}{Lemma}



\title{%
The Arithmetic Cosine Transform: Exact and Approximate Algorithms
}

\author{%
R.~J.~Cintra%
\thanks{%
R.~J.~Cintra is with 
the Signal Processing Group,
Departamento de Estat\'{\i}stica, 
Universidade Federal de Pernambuco.
This work was done while visiting 
the Department of Electrical and Computer Engineering, University of Calgary, Calgary, AB, Canada.
E-mail: 
\protect\url{rjdsc@stat.ufpe.org},
}
\and
V. S. Dimitrov%
\thanks{V. S. Dimitrov is with the
Advanced Technology Information Processing Systems (ATIPS) Laboratory, 
University of Calgary, Calgary, AB, Canada.
This work was done
during second author's sabbatical leave at Nanyang Technological University,
Singapore.
Email: \url{vdvsd103@gmail.com}}
}

\date{}

\newcommand{\myabstract}{%
In this paper,
we introduce 
a new class of transform method --- the arithmetic cosine transform (ACT).
We provide the central mathematical properties of the ACT, 
necessary in designing efficient and accurate implementations of the new transform method.
The key mathematical tools used in the paper come from analytic
number theory, in particular the properties of the Riemann zeta function.
Additionally,
we demonstrate that an exact signal interpolation is achievable for any
block-length.
Approximate calculations were also considered.
The numerical
examples provided show the potential of the ACT for various digital
signal processing applications.
}

\newcommand{\mykeywords}{%
Discrete cosine transform,
arithmetic transform algorithms,
nonuniform sampling

}

\begin{document}

\makeatletter
\if@twocolumn

\twocolumn[%
  \maketitle
  \begin{onecolabstract}
    \myabstract
  \end{onecolabstract}
  \begin{center}
    \medskip
    \mykeywords
  \end{center}
  \bigskip
]
\saythanks

\else

  \maketitle
  \begin{abstract}
    \myabstract
  \end{abstract}
  \begin{center}
    \medskip
    \mykeywords
  \end{center}
  \bigskip
  \onehalfspacing
\fi

\section{Introduction}

The arithmetic Fourier transform (AFT) has been emerged in 1988 as
a signal processing tool~\cite{tufts1988arithmetic}.
In a broad sense,
the AFT is a number-theoretic algorithm for spectrum evaluation.
One of the main features of the AFT is
its virtually multiplication-free nature.
Although initial versions of the AFT procedure
could only compute the real part of DFT~\cite{tufts1988arithmetic,tufts1989image},
additional improvements due to Reed \emph{et~al.}
expanded the AFT methodology and
the DFT could be fully evaluated~\cite{reed1992vlsi,reed1990mobius}.

Arithmetic methods lack a larger degree of adoption mainly
because data is required to be sampled nonuniformly.
In fact, the AFT algorithm imposes a sampling scheme
that can be related to Farey fractions~\cite{tepedelenlioglu1989note}.
This can be a sensitive issue since
incoming signals are usually uniformly sampled.
Therefore,
in order to obtain the necessary nonuniform samples,
literature essentially suggests two methods:
(i) signal oversampling
or
(ii) sample interpolation.
On the one hand,
oversampling can be a major drawback
because it requires an above Nyquist rate sampling.
Thus, this is often discarded as an option~\cite{reed1992vlsi}.
On the other hand, 
to obtain nonuniform samples from uniformly sampled data,
interpolation methods have been considered.
Archived literature lacks an exact procedure to interpolate
nonuniform samples obtained from arithmetic transform methods. 
Therefore,
in a simplistic way,
zero-order approximations and linear interpolation methods
have been considered~\cite{kelley1993efficient}.
Although not furnishing exact computations,
these crude interpolation methods could attain
acceptable approximations,
when large block-lengths were considered~\cite{reed1990mobius}.
However,
for small block-lengths,
the implied interpolation errors could be large enough
to totally preclude a meaningful computation.

Nevertheless,
recent existing works have applied the AFT algorithm in a number of ways.
In~\cite{lima2004dtmf}, 
the AFT is considered as an alternative to
the Goertzel algorithm for the single spectral component evaluation.
The AFT could also provide frequency domain testing units for
built-in self-test routines with reduced hardware overhead~\cite{venkateshwaran2004testing}.
Additionally, hardware considerations have been directed to enhance
the associated nonuniform sampling of the AFT~\cite{ray2007conversion}.
Finally,
the AFT has been considered as a tool for DCT evaluation~\cite{tan2007cmos}.

In all above mentioned applications,
the AFT procedure as described in~\cite{tufts1988arithmetic,reed1992vlsi}
was considered.
In particular,
even when the DCT was required as in~\cite{tan2007cmos},
it was obtained by means of the AFT~\cite{tufts1989image}.
Indeed, 
the DFT spectrum
can be mapped into
the DCT spectrum at the cost of extra computations.

The goal of this paper is twofold.
First, 
we aim to introduce a class of purely arithmetic algorithms
for the DCT,
termed arithmetic cosine transforms (ACT).
Such new methods are the main contribution of this work.
To accomplish this objective,
we examine arithmetic averages based on
the existing arithmetic Fourier transform.
Then,
we advance new arithmetic averages specially
tailored for the DCT computation.

Second,
an exact interpolation method
for the ACT is sought.
Such exact interpolation could provide solid mathematical grounds 
for arithmetic transform methods,
constituting a theoretical advance in the area.
In fact, any arithmetic transform method could only furnish an exact computation
if a precise interpolation method could be devised.
In particular,
we discuss an interpolation scheme
that could provide an exact spectrum evaluation,
even when short block-lengths are considered.
Moreover,
an asymptotic analysis for
the proposed interpolation approach is
elaborated
and
a heuristic approximation algorithm is devised.

The article unfolds as follows.
In Section~\ref{section.act},
we examine some of the existing tools employed
in the AFT theory
and
propose a new algorithm for the DCT.
Section~\ref{section.interpolation} is devoted to the 
interpolation issues that
arise from the introduced ACT.
In Section~\ref{section.matrix},
the discussed arithmetic transform is
formatted in terms of matrix representation.
Conclusions and final remarks
are given in Section~\ref{section.conclusions}.

\section{Arithmetic cosine transform}
\label{section.act}

\subsection{Preliminary concepts}

Among the several types of DCTs,
we separated the DCT-II,
which can be regarded as the most employed form~\cite{britanak2007cosine}.
This transformation relates two $N$-dimensional vectors,
$\mathbf{v}$ and $\mathbf{V}$,
according to 
the following pair of relations~\cite{britanak2007cosine}
\begin{align}
\label{eq.definition.dct}
V_k
&=
\sqrt{\frac{2}{N}}
\alpha_k
\sum_{n=0}^{N-1}
v_n
\cos
\left(
\frac{\pi k (n + 1/2)}{N}
\right),
\quad
k = 0, 1, \ldots, N-1,
\end{align}

\begin{align*}
v_n
&=
\sqrt{\frac{2}{N}}
\sum_{k=0}^{N-1}
\alpha_k
V_k
\cos
\left(
\frac{\pi k (n + 1/2)}{N}
\right),
\quad
n = 0, 1, \ldots, N-1,
\end{align*}
where the coefficients $\alpha_k$, $k=0,1,\ldots,N-1$
are given by
\begin{align*}
\alpha_k 
=
\begin{cases}
1/\sqrt{2}, & \text{if $k=0$}, \\
1, & \text{otherwise.} 
\end{cases}
\end{align*}
Hereafter,
we refer to this transformation simply as DCT.

Before defining and examining the ACT,
some fundamental tools are necessary.
In what follows,
we introduce the generalized M\"obius inversion formula tailored
for finite series.
Additionally,
some useful identities for trigonometric sums are also presented.
We state these preambulary results below.

\begin{theorem}
[Generalized M\"obius Inversion Formula for Finite Series]
\label{theorem.generalized}
Let $\{f_n\}$ be a sequence (e.g., signal samples) such that
it is nonnull for 
$1\leq n \leq N$ and
null for $n> N$.
Admit another sequence $\{g_n\}$ defined
as
\begin{align*}
g_n
=
\sum_{k=1}^{\lfloor N /n \rfloor}
a_k
f_{kn},
\end{align*}
where $\{a_n\}$ is a sequence of scalar coefficients
and
$\lfloor \cdot \rfloor$ is the floor function.
Then,
\begin{align*}
f_n
=
\sum_{m=1}^{\lfloor N /n \rfloor}
b_m
g_{mn},
\end{align*}
where $\{ b_n \}$
is the Dirichlet inverse sequence of $\{a_n\}$,
given that it exists~\cite{bateman2004analytic}.
\end{theorem}
\begin{proof}
It follows by
applying the arguments
described in the proof of Theorem 3 from~\cite[p.~469]{reed1990mobius}
into
the M\"obius inversion formula as shown in~\cite[p.~556]{goldberg1956inversion}.
\end{proof}
If 
we consider the unitary sequence
$a_n =1$, for $n =1, 2, 3, \ldots$,
then
the associated Dirichlet inverse sequence
is the M\"obius sequence
$b_n = \mu(n)$, for $n=1,2,3,\ldots$
The M\"obius function $\mu(n)$~\cite{apostol1984analytic}
is defined over the positive integers and
is given by
\begin{align*}
\mu(n)
&=
\begin{cases}
1, & \text{if $n=1$,} \\
(-1)^q, & \text{if $n$ can be factorized into $q$ distinct primes,} \\
0, & \text{if $n$ is divisible by a square number.}
\end{cases}
\end{align*}
In this particular case,
arithmetic transform literature simply refers to 
the above theorem
as
the M\"obius inversion formula for finite series~\cite{reed1990mobius,reed1992vlsi}.

The following lemmas are pivotal for subsequent developments.
They link usual trigonometric functions to number theoretic behavior functions,
a connection that is not always trivial~\cite{pei2007ramanujan,laohakosol2006ramanujan}.

\begin{lemma}[Reed's Lemma]
\label{lemma.sum}
Let $k>0$ and $k'\geq0$ be positive integers.
Then,
\begin{align*}
\sum_{m=0}^{k-1}
\cos
\left(
2m \frac{k'}{k} \pi
\right)
=
\begin{cases}
k, & \text{if $k|k'$}, \\
0,  & \text{otherwise,}
\end{cases}
\end{align*}
and
\begin{align*}
\sum_{m=0}^{k-1}
\sin
\left(
2m \frac{k'}{k} \pi
\right)
=
0
.
\end{align*}
\end{lemma}

\noindent
We amplified Lemma~\ref{lemma.sum} with the next two results.

\begin{lemma}
\label{lemma.2}
Let $k>0$ and $k'\geq0$ be positive integers.
Then,
\begin{align*}
\sum_{m=0}^{2k-1}
\cos
\left(
\pi m \frac{k'}{k}
\right)
=
\begin{cases}
2k, & \text{if $2k|k'$}, \\
0, & \text{otherwise.}
\end{cases}
\end{align*}
\end{lemma}

\begin{proposition}
\label{prop.sum}
Let $k>0$ and $k'\geq0$ be integers and $\alpha$ be a real quantity.
Then,
\begin{align*}
\sum_{m=0}^{k-1}
\cos
\left(
2\pi \frac{k'}{k}(m+\alpha)
\right)
=
\begin{cases}
k, & \text{if $k'=0$,} \\
k \cos(2\pi \frac{k'}{k} \alpha), & \text{if $k | k'$, $k'\neq 0$,} \\
0, & \text{otherwise.}
\end{cases}
\end{align*}
\begin{proof}
It follows from usual trigonometric manipulations
and an application of Lemma~\ref{lemma.sum}.
\end{proof}
\end{proposition}

\subsection{Arithmetic averages}

Based on existing definitions and concepts 
inherited from the AFT theory,
our initial attempt to define an
arithmetic transform procedure for the DCT
is detailed below.
Thus,
we consider the following definition
for AFT-like averages.

\begin{definition}[$k$th AFT-like average]
\label{def.bruns}
Let the $k$th average be defined as
\begin{align*}
S_k
\triangleq
\frac{1}{2k}
\sum_{m=0}^{2k-1}
v_{m \frac{N}{k}-\frac{1}{2}},
\quad
k = 1, 2, \ldots, N-1.
\end{align*}
\end{definition}

Except for the $1/2$-shift, 
which is naturally imposed by the DCT kernel,
the above concept was previously described in~\cite{fisher1989vlsi,reed1990mobius,reed1992vlsi}
as a framework for the AFT.

Let us investigate the consequences of this definition.
Relaxing the integer index constraint
of
the DCT formulae and
substituting the inverse DCT relation 
into the above described $k$th average,
we must have
\begin{align*}
S_k
&=
\frac{1}{2k}
\sum_{m=0}^{2k-1}
v_{m \frac{N}{k}-\frac{1}{2}}
\\
&=
\frac{1}{2k}
\sum_{m=0}^{2k-1}
\sqrt{\frac{2}{N}}
\sum_{k'=0}^{N-1}
\alpha_{k'}
V_{k'}
\cos
\left(
\frac{\pi k' (m \frac{N}{k} -\frac{1}{2} + \frac{1}{2})}{N}
\right)
\\
&=
\sqrt{\frac{2}{N}}
\frac{1}{2k}
\sum_{m=0}^{2k-1}
\left(
\gamma V_0
+
\sum_{k'=0}^{N-1}
V_{k'}
\cos
\left(
\pi  m \frac{k'}{k}
\right)
\right)
\\
&=
\sqrt{\frac{2}{N}}
\gamma V_0
+
\sqrt{\frac{2}{N}}
\frac{1}{2k}
\sum_{m=0}^{2k-1}
\sum_{k'=0}^{N-1}
V_{k'}
\cos
\left(
\pi  m \frac{k'}{k}
\right),
\quad
k = 1, 2, \ldots, N-1,
\end{align*}
where $\gamma = 1/\sqrt{2} - 1$.
By interchanging the summations, it holds that
\begin{align}
\label{eq.30}
S_k
&=
\sqrt{\frac{2}{N}}
\gamma V_0
+
\sqrt{\frac{2}{N}}
\frac{1}{2k}
\sum_{k'=0}^{N-1}
V_{k'}
\sum_{m=0}^{2k-1}
\cos
\left(
\pi  m \frac{k'}{k}
\right),
\quad
k = 1, 2, \ldots, N-1.
\end{align}
Invoking Lemma~\ref{lemma.2},
we have
\begin{align*}
S_k
&=
\sqrt{\frac{2}{N}}
\gamma V_0
+
\sqrt{\frac{2}{N}}
\frac{1}{2k}
\sum_{k'=0}^{N-1}
V_{k'}
\left\{
\begin{array}{cl}
2k, & \text{if $2k|k'$}, \\
0, & \text{otherwise,}
\end{array}
\right\}
\\
&=
\sqrt{\frac{2}{N}}
\gamma V_0
+
\sqrt{\frac{2}{N}}
\sum_{k'=0}^{N-1}
V_{k'}
\left\{
\begin{array}{cl}
1, & \text{if $2k|k'$}, \\
0, & \text{otherwise,}
\end{array}
\right\}
,
\quad
k = 1, 2, \ldots, N-1.
\end{align*}
Let $k'= 2sk$.
Then,
\begin{align*}
S_k
&=
\sqrt{\frac{2}{N}}
\gamma V_0
+
\sqrt{\frac{2}{N}}
\sum_{s=0}^{\left\lfloor\frac{N-1}{2k}\right\rfloor}
V_{2sk}
,
\quad
k = 1, 2, \ldots, N-1.
\end{align*}
Without any loss of generality,
let us assume that input signals have a null mean value.
Therefore, it follows that $V_0$ is zero.
This assumption
does not affect
the value of
the remaining components
of the DCT spectrum.
Thus, we find that
\begin{align*}
S_k
&=
\sqrt{\frac{2}{N}}
\sum_{s=1}^{\left\lfloor\frac{N-1}{2k}\right\rfloor}
V_{2sk}
=
\sqrt{\frac{2}{N}}
\sum_{s=1}^{\left\lfloor\frac{N-1}{2k}\right\rfloor}
\left\{
\begin{array}{cl}
0, & \text{if $s$ is odd,} \\
1, & \text{otherwise,}
\end{array}
\right\}
V_{sk}
,
\quad
k = 1, 2, \ldots, N-1.
\end{align*}
In order to invert
this last expression,
we must obtain the
Dirichlet inverse of the sequence
$\{0, 1, 0, 1, 0, 1, \ldots \}$,
as required by Theorem~\ref{theorem.generalized}.
However,
a sequence can only admit the Dirichlet inverse 
if and only if its first element is nonnull~\cite[p.~18]{bateman2004analytic}.
This is clearly not the case for the particular sequence under examination.

We conclude that the derivation of an arithmetic method for the DCT
requires more than simply reusing the AFT averages.
For the DCT computation,
specifically tailored averages should be considered.
In the next section,
such specific derivation is sought.
However,
the mathematical manipulations and arguments shown above
prove to be a useful framework.
We show next that the proposition of DCT specific averages
can be greatly benefited from the above exposition.

\subsection{ACT averages}

In order
to derive the ACT,
we adjusted the AFT averages as
suggested in the following definition.

\begin{definition}[$k$th ACT average]
\label{def.act.average}
Let the $k$th ACT average be defined as
\begin{align*}
S_k
\triangleq
\frac{1}{k}
\sum_{m=0}^{k-1}
v_{2(m+\beta) \frac{N}{k}-\frac{1}{2}},
\quad
k = 1, 2, \ldots, N-1,
\end{align*}
where $\beta$ is a fixed real number.
\end{definition}

Considering 
algebraic manipulations similar to the ones
that has led us from Definition~\ref{def.bruns} 
to equation~(\ref{eq.30})
and
admitting that the input signal has null mean,
\emph{mutatis mutandis},
we obtain
\begin{align*}
S_k
&=
\sqrt{\frac{2}{N}}
\frac{1}{k}
\sum_{k'=1}^{N-1}
V_{k'}
\sum_{m=0}^{k-1}
\cos
\left(
2\pi  (m+\beta) \frac{k'}{k}
\right),
\quad
k = 1, 2, \ldots, N-1.
\end{align*}
Thus,
invoking Proposition~\ref{prop.sum},
we have
\begin{align*}
S_k
&= 
\sqrt{\frac{2}{N}}
\sum_{k'=1}^{N-1}
V_{k'}
\left\{
\begin{array}{cl}
\cos(2\pi \frac{k'}{k} \beta), & \text{if $k | k'$,} \\
0, & \text{otherwise,}
\end{array}
\right\}
,
\quad
k = 1, 2, \ldots, N-1.
\end{align*}
Performing the substitution $k'=ks$,
it follows that
\begin{align*}
S_k
&=
\sqrt{\frac{2}{N}}
\sum_{s=1}^{\left\lfloor\frac{N-1}{k}\right\rfloor}
\cos(2\pi s \beta)
V_{sk}
,
\quad
k = 1, 2, \ldots, N-1.
\end{align*}
Observe that the above expression
is suitable for the application of
the generalized M\"obius inversion formula for finite series.
Under the notation of
Theorem~\ref{theorem.generalized},
we recognize
$a_n = \cos(2\pi n \beta)$,
$n = 1, 2, 3, \ldots$

Incidentally,
not always the
Dirichlet inverse of $\{a_n\}$
is well-defined.
Only when $a_1 \neq 0$,
the existence of the Dirichlet inverse
can be considered~\cite{bateman2004analytic}.
Thus,
we must impose
$\cos(2\pi\beta) \neq 0$
as a necessary condition for the derivation of the
ACT procedure.
This issue is precisely the point that
Definition~\ref{def.bruns} misses.

However,
finding the Dirichlet inverse of $\{a_n\}$,
say $\{ b_n \}$
for arbitrary values of $\beta$
can be a cumbersome maneuver.
Therefore,
we separated two particular useful cases:
(i) $\beta = 0$
and
(ii) $\beta = 1/2$.
Notice that $\beta = 1/4$ leads to a non-invertible
sequence, 
since this makes $a_1 = 0$.

For $\beta=0$,
we have the unitary sequence
$a_n =1$ and $b_n = \mu(n)$,
for $n =1, 2, 3, \ldots$
This is usually the
situation
addressed in standard AFT analysis.
On the other hand,
setting $\beta = 1/2$ yields
$a_n = (-1)^n$,
$n = 1, 2, 3, \ldots$
In this case,
the Dirichlet inverse is not immediately
recognized,
but is can be obtained analytically.
In the Appendix,
we derive the sought Dirichlet inverse,
which is given by
\begin{align}
\label{eq.b_n}
b_n
=
\begin{cases}
- \mu(n), & \text{if $n$ is odd,} \\
- 2^{m-1} \mu(2^{-m} n), & \text{if $n = 2^m s$, where $s$ is odd.}
\end{cases}
\end{align}
The first 32 terms of $\{b_n\}$ are listed below:
\begin{align*}
\begin{array}{rrrrrrrr}
          -1,&    -1,&     1,&    -2,&     1,&     1,&     1,&    -4, \\
           0,&     1,&     1,&     2,&     1,&     1,&    -1,&    -8, \\
           1,&     0,&     1,&     2,&    -1,&     1,&     1,&     4, \\
           0,&     1,&     0,&     2,&     1,&    -1,&     1,&   -16.
\end{array}
\end{align*}
In the context of digital signal processing
this is a potentially useful sequence,
since 
multiplying a given number by a power-of-two 
can be implemented by simple bit shift operations,
which possess a low computational complexity.

Regardless the choice of $\beta$,
to invert the following expression
given by
\begin{align*}
S_k
&=
\sqrt{\frac{2}{N}}
\sum_{s=1}^{\left\lfloor\frac{N-1}{k}\right\rfloor}
a_s
V_{sk}
,
\quad
k = 1, 2, \ldots, N-1,
\end{align*}
we may consider an auxiliary sequence given by $G_k = \sqrt{2/N} V_k$ for all $k$.
Thus,
a direct application of the generalized M\"obius inversion formula for finite series 
just tells us that
\begin{align*}
G_k &=
\sum_{l=1}^{\left\lfloor\frac{N-1}{k}\right\rfloor}
b_l S_{kl},
\quad
k = 1, 2, \ldots, N-1,
\end{align*}
where $\{b_n\}$
is the
Dirichlet inverse of $\{a_n\}$.
Finally, 
undoing the auxiliary substitution
we can write
\begin{align}
\label{eq.V.S}
V_k &=
\sqrt{\frac{N}{2}}
\sum_{l=1}^{\left\lfloor\frac{N-1}{k}\right\rfloor}
b_l
S_{kl},
\quad
k = 1, 2, \ldots, N-1.
\end{align}
Depending on whether 
$\{a_n \}$ is the unitary sequence 
or
the alternate sequence $-1, 1, -1, 1, \ldots$,
the inverse sequence $\{b_n\}$ is the M\"obius sequence or
the sequence described in~(\ref{eq.b_n}),
respectively.

Recognizing that unitary and M\"obius sequences
constitute a simpler pair of Dirichlet inverse sequences,
in the rest of this paper
we focus our attention to the case $\beta = 0$.
However,
all ensuing developments encompass
the proposed alternative formulation ($\beta = 1/2$)
without significant modifications.
Nevertheless,
the case $\beta = 1/2$ can furnish a link between
Definitions~\ref{def.bruns} and~\ref{def.act.average}.
Indeed,
for $\beta=1/2$, we have
\begin{align*}
S_k
&=
\frac{1}{k}
\sum_{m=0}^{k-1}
v_{2(m+\frac{1}{2}) \frac{N}{k}-\frac{1}{2}}
\\
&=
\frac{1}{k}
\sum_{m=0}^{k-1}
v_{(2m+1) \frac{N}{k}-\frac{1}{2}}
\\
&=
\frac{1}{k}
\mathop{\sum_{m=1}^{2k-1}}_{\text{$m$ odd}}
v_{m \frac{N}{k}-\frac{1}{2}},
\quad
k = 1, 2, \ldots, N-1.
\end{align*}
This last summation corresponds to the average shown in
Definition~\ref{def.bruns},
when only odd values of the dummy index are considered.
In other words,
if the samples required by Definition~\ref{def.bruns}
are submitted to a downsampling by a factor of 2,
then Definition~\ref{def.bruns} collapses to
Definition~\ref{def.act.average} for $\beta = 1/2$.
In a sense,
this observation establishes a connection between both formalisms.

As far as the computational complexity of the
M\"obius inversion formulae
are concerned, 
we can provide the following probabilistic reasoning.
The probability that a randomly chosen integer is not divisible
by a perfect square is $6/\pi^2 \approx 0.61$~\cite[p.~4]{bateman2004analytic}.
Therefore,
61\% of the values of the M\"obius function are zeros; 
meaning that
the computation of $G_k$ (or $V_k$) requires
$(1 - 6/\pi^2) \left\lfloor (N-1)/k\right\rfloor$
additions/subtractions on average.

Up to this point,
we assumed that input signals have null mean.
Now, let us remove this restriction.
Let $\bar{v} \triangleq \frac{1}{N}\sum_{n=0}^{N-1}v_n$
be the mean value of the input signal.
Thus, an arbitrary signal $\mathbf{v}'$
can be converted into a null mean signal
by simply subtracting the quantity $\bar{v}$.
Therefore,
the $k$th ACT averages $S_k'$
associated to $\mathbf{v}' - \bar{v}$
can be manipulated as follows:
\begin{align*}
S_k'
&\triangleq
\frac{1}{k}
\sum_{m=0}^{k-1}
(v_{2m \frac{N}{k}-\frac{1}{2}} - \bar{v})
\\
&
=
\frac{1}{k}
\left(
\sum_{m=0}^{k-1}
v_{2m \frac{N}{k}-\frac{1}{2}}
\right)
-\bar{v}
\\
&
=
S_k
-\bar{v}
,
\quad
k = 1, 2, \ldots, N-1.
\end{align*}
Consequently,
equation~(\ref{eq.V.S}) can be rearranged,
yielding
\begin{align*}
V_k 
&=
\sqrt{\frac{N}{2}}
\sum_{l=1}^{\left\lfloor\frac{N-1}{k}\right\rfloor}
\mu(l) S_{kl}' 
\\
&=
\sqrt{\frac{N}{2}}
\sum_{l=1}^{\left\lfloor\frac{N-1}{k}\right\rfloor}
\mu(l) 
(S_{kl} -\bar{v}) 
\\
&=
\sqrt{\frac{N}{2}}
\left(
\sum_{l=1}^{\left\lfloor\frac{N-1}{k}\right\rfloor}
\mu(l) 
S_{kl}
-
\sum_{l=1}^{\left\lfloor\frac{N-1}{k}\right\rfloor}
\mu(l) 
\bar{v} 
\right)
,
\quad
k = 1, 2, \ldots, N-1.
\end{align*}

Considering the Mertens function $M(n)$,
which is defined as 
$M(n) \triangleq \sum_ {m=1}^n \mu(m)$~\cite[p.~272]{schroeder2008number},
we obtain a more compact expression as
\begin{align}
\label{eq.with.Mertens}
V_k 
&=
\sqrt{\frac{N}{2}}
\left(
\sum_{l=1}^{\left\lfloor\frac{N-1}{k}\right\rfloor}
\mu(l) 
S_{kl}
\right)
-
\sqrt{\frac{N}{2}}
\bar{v} 
M\left(\left\lfloor\frac{N-1}{k}\right\rfloor\right)
,
\quad
k = 1, 2, \ldots, N-1.
\end{align}
The zeroth component $V_0$ is expressed separately by
$V_0 = \sqrt{N} \bar{v}$.

\section{ACT interpolation}
\label{section.interpolation}

\subsection{Fractional indices manipulation}

Now we must address the problem of handling 
fractional index samples.
Usually,
arithmetic transform literature addresses this issue
by prescribing the utilization of zero- or first-order interpolation methods.
We argue that such a naive approach is not the proper way of interpolating.

Again,
let us relax the integer index constraint
of the DCT formulae.
Considering a possibly noninteger quantity $r$,
the following construction
is derived given by
\begin{align*}
v_r
&=
\sqrt{\frac{2}{N}}
\sum_{k=0}^{N-1}
\alpha_k
V_k
\cos
\left(
\frac{\pi k (r + 1/2)}{N}
\right).
\end{align*}
Taking into account the direct transformation formula,
we obtain
\begin{align*}
v_r
&=
\sqrt{\frac{2}{N}}
\sum_{k=0}^{N-1}
\alpha_k
\left(
\sqrt{\frac{2}{N}}
\alpha_k
\sum_{n=0}^{N-1}
v_n
\cos
\left(
\frac{\pi k (n + 1/2)}{N}
\right)
\right)
\cos
\left(
\frac{\pi k (r + 1/2)}{N}
\right).
\end{align*}
Inverting the summation order yields
\begin{align*}
v_r
&=
\frac{2}{N}
\sum_{n=0}^{N-1}
v_n
\sum_{k=0}^{N-1}
\alpha_k^2
\cos
\left(
\frac{\pi k (n + 1/2)}{N}
\right)
\cos
\left(
\frac{\pi k (r + 1/2)}{N}
\right).
\end{align*}
Then, let us define the ACT weighting function as
\begin{align*}
w_n(r)
&\triangleq
\frac{2}{N}
\sum_{k=0}^{N-1}
\alpha_k^2
\cos
\left(
\frac{\pi k (n + 1/2)}{N}
\right)
\cos
\left(
\frac{\pi k (r + 1/2)}{N}
\right)
\\
&=
-
\frac{1}{N}
+
\frac{2}{N}
\sum_{k=0}^{N-1}
\cos
\left(
\frac{\pi k (n + 1/2)}{N}
\right)
\cos
\left(
\frac{\pi k (r + 1/2)}{N}
\right)
,
\quad
n=0,1,\ldots,N-1
.
\end{align*}
Thus, 
the samples associated to fractional indices 
can be obtained after
a linear combination of
the available uniformly obtained samples.
Hence,
\begin{align}
\label{eq.interpolation}
v_r
&=
\sum_{n=0}^{N-1}
w_n(r)
v_n
.
\end{align}
Notice also that if $r$ is integer, we have that
\begin{align*}
w_n(r) =
\begin{cases}
1, & \text{if $n = r$,} \\
0, & \text{otherwise.}
\end{cases}
\end{align*}
This fact stems from the orthogonality properties of the
transformation kernel.

We further investigate the suggested weighting function.
The next proposition states that the discussed weighting functions
are indeed inherently normalized.

\begin{proposition}
\label{prop.w.sums.1}
The ACT weighting function satisfies
the following summation formula
\begin{align*}
\sum_{n=0}^{N-1}
w_n(r)
=
1.
\end{align*}
\begin{proof}
Usual trigonometric manipulations furnish
the derivations below:
\begin{align*}
\sum_{n=0}^{N-1}
w_n(r)
&=
\sum_{n=0}^{N-1}
\left(
-
\frac{1}{N}
+
\frac{2}{N}
\sum_{k=0}^{N-1}
\cos
\left(
\frac{\pi k (n + 1/2)}{N}
\right)
\cos
\left(
\frac{\pi k (r + 1/2)}{N}
\right)
\right)
\\
&=
-1
+
\frac{2}{N}
\sum_{k=0}^{N-1}
\cos
\left(
\frac{\pi k (r + 1/2)}{N}
\right)
\sum_{n=0}^{N-1}
\cos
\left(
\frac{\pi k (n + 1/2)}{N}
\right)
.
\end{align*}
However,
for each $k$,
the inner summation can be expanded as follows:
\begin{align*}
\sum_{n=0}^{N-1}
\cos
\left(
\frac{\pi k (n + 1/2)}{N}
\right)
&=
\cos
\left(
\frac{\pi k}{2N}
\right)
\sum_{n=0}^{N-1}
\cos
\left(
\frac{2\pi (k/2) n}{N}
\right)
-
\sin
\left(
\frac{\pi k}{2N}
\right)
\sum_{n=0}^{N-1}
\sin
\left(
\frac{2\pi (k/2) n}{N}
\right)
.
\end{align*}
Since dummy index $k$ runs from $0$ to $N-1$,
we have that $N$ never divides $k/2$.
Therefore,
applying Lemma~\ref{lemma.sum},
we obtain that
\begin{align*}
\sum_{n=0}^{N-1}
\cos
\left(
\frac{\pi k (n + 1/2)}{N}
\right)
&=
\cos
\left(
\frac{\pi k}{2N}
\right)
\left\{
\begin{array}{cl}
N, & \text{if $k=0$}, \\
0, & \text{otherwise,}
\end{array}
\right\}
\\
&=
\left\{
\begin{array}{cl}
N, & \text{if $k=0$}, \\
0, & \text{otherwise.}
\end{array}
\right.
\end{align*}
Finally,
returning to the previous double summation,
we establish that
\begin{align*}
\sum_{n=0}^{N-1}
w_n(r)
&=
-1
+
\frac{2}{N}
\sum_{k=0}^{N-1}
\cos
\left(
\frac{\pi k (r + 1/2)}{N}
\right)
\left\{
\begin{array}{cl}
N, & \text{if $k=0$}, \\
0, & \text{otherwise,}
\end{array}
\right\}
\\
&=
1
.
\end{align*}
\end{proof}
\end{proposition}

Regardless of the considered block-length,
the use of the ACT interpolation
results in an exact calculation of the DCT spectrum.
Indeed,
no approximation was considered in any of our arguments.

\subsection{An example: the $8$-point ACT}

To illustrate the ACT structure and its interpolation scheme,
we devised an example.
Intentionally,
we selected the short block-length transformation
furnished by 
the $8$-point DCT.
This particular block-length is
widely adopted in several image and video coding standards,
such as
JPEG, MPEG-1, MPEG-2, H.261, and H.263~\cite{roma2007hybrid}.
The $8$-point DCT is also subject to an extensive analysis in~\cite{britanak2007cosine}.
In the following,
we set $\beta = 0$.

The first step of the ACT procedure consists of
the identification of the necessary interpolation points.
According to Definition~\ref{def.act.average},
these points are given
by
$2m \frac{8}{k}-\frac{1}{2}$
for $k = 1, 2, \ldots, 7$ 
and
$m=0,1,2, \ldots k-1$.
Therefore,
we find the following
fractional indices:
$r \in 
\left\{
-\frac{1}{2},
\frac{25}{14},
\frac{13}{6},
\frac{27}{10},
\frac{7}{2},
\frac{57}{14},
\frac{29}{6},
\frac{59}{10},
\frac{89}{14},
\frac{15}{2}
\right\}
$.
Fig{.}~\ref{figure.act8}
depicts a block diagram of the full algorithm.
The ACT interpolation block
calculates
the required samples $v_r$
according to 
the discussed interpolation procedure.

\begin{figure}
\centering
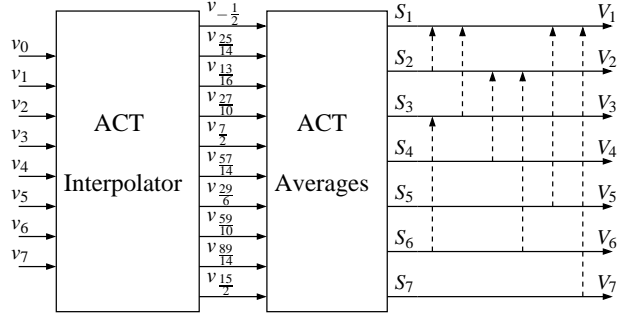
\caption{Block diagram of the $8$-point ACT. Dashed lines indicate multiplication by $-1$.}
\label{figure.act8}
\end{figure}

The fractional index samples are then
employed to obtain the ACT averages.
The block of ACT averages simply implements
the following set of equations:
\begin{align*}
S_1 &= v_{-\frac{1}{2}} \\
2S_2 &= S_1 + v_{\frac{15}{2}} \\
3S_3 &= S_1 + 2 v_{\frac{29}{6}} \\
4S_4 &= S_2 + 2 v_{\frac{7}{2}} \\
5S_5 &= S_1 + 2 v_{\frac{27}{10}} + 2v_{\frac{59}{10}} \\
6S_6 &= S_2 + S_3 + 2v_{\frac{13}{6}} \\
7S_7 &= S_1 + 2 v_{\frac{25}{14}} + 2v_{\frac{57}{14}} + 2v_{\frac{89}{14}}.
\end{align*}
Finally,
the ACT averages are combined
with respect to the M\"obius function (cf.~(\ref{eq.V.S})).
The resulting calculation involves no approximations and furnishes the exact DCT spectrum.

\subsection{Asymptotic analysis}

In this section,
we closely examine the weighting function required by
the interpolation procedure.
We aim to propose
simpler interpolation expressions
allowing an efficient computation of the ACT.

First,
we note that the ACT weighting function can be formulated in
closed form as detailed below.
In fact,
invoking elementary trigonometric identities,
we can establish the following relations:
\begin{align*}
w_n(r)
&=
-
\frac{1}{N}
+
\frac{2}{N}
\sum_{k=0}^{N-1}
\cos
\left(
\frac{\pi k (n + 1/2)}{N}
\right)
\cos
\left(
\frac{\pi k (r + 1/2)}{N}
\right)
\\
 &=
 -
 \frac{1}{N}
 +
 \frac{1}{N}
 \sum_{k=0}^{N-1}
 \left(
 \cos
 \left(
 \frac{\pi k (n+r+1)}{N}
 \right)
 +
 \cos
 \left(
 \frac{\pi k (n-r)}{N}
 \right)
 \right)
 \\
&=
-
\frac{1}{N}
+
\frac{1}{N}
\sum_{k=0}^{N-1}
\cos
\left(
\frac{\pi k (n+r+1)}{N}
\right)
+
\frac{1}{N}
\sum_{k=0}^{N-1}
\cos
\left(
\frac{\pi k (n-r)}{N}
\right),
\quad
n=0,1,\ldots,N-1
.
\end{align*}
The above trigonometric summations 
can be given in terms of the 
Dirichlet kernel~\cite[p.~312]{krantz2005real}.
Therefore,
it holds that
\begin{align*}
w_n(r)
&=
-\frac{1}{N}
+
\frac{1}{N}
\left(
\frac{1}{2}
+
\frac{1}{2}
D_{N-1}\left( \frac{\pi}{N} (n+r+1)\right)
\right)
+
\frac{1}{N}
\left(
\frac{1}{2}
+
\frac{1}{2}
D_{N-1}\left( \frac{\pi}{N} (n-r)\right)
\right)
\\
&=
\frac{1}{2N}
\left(
D_{N-1}\left( \frac{\pi}{N} (n+r+1)\right)
+
D_{N-1}\left( \frac{\pi}{N} (n-r)\right)
\right),
\end{align*}
where
$D_N(x) = \frac{\sin((N+1/2)x)}{\sin(x/2)}$
denotes the Dirichlet kernel.

The similarity between
the Dirichlet kernel
and the $\mathrm{sinc}$ function is apparent.
Indeed,
as $N\to\infty$,
these functions are interchangeable in several situations~\cite{bachman1999analysis}.
However,
even for small values of $N$
such an approximation is good~\cite[p.~180]{bachman1999analysis}. 
For instance, taking $N=8$ and
centered functions at the origin,
it follows that
the mean square error (MSE)
of
implied approximation is less than $2 \times 10^{-3}$
over the interval $[-N/2, N/2]$.
Thus,
the discussed weighting function assumes a
limiting form given by 
the sum of two translated sampling functions as
\begin{align*}
\lim_{N\to\infty}
w_n(r)
=
\mathrm{sinc}(n+r+1)
+
\mathrm{sinc}(n-r)
,
\end{align*}
where
$\mathrm{sinc}(x) \triangleq \frac{\sin(\pi x)}{\pi x}$.

Notice that this asymptotic expression connects 
the ACT interpolation to the $\mathrm{sinc}$ function interpolation.
Signal interpolation according to the $\mathrm{sinc}$ function
can be efficiently implemented in the time domain~\cite{schanze1995sinc,dooley2000notes}.
Additionally,
fractional delay FIR filtering methods 
offer another computational approach~\cite{valimaki2001fractional,laakso1996splitting}.

The
limiting form of $w_n(r)$
leads us
to draw some additional conclusions 
on its asymptotic behavior.
Let $[\cdot]$ denote the nearest integer function as implemented in C or Matlab programming languages
and
$r$ be a fractional interpolation point.
We observed that, 
when $0 \leq [r] \leq N-1$,
the asymptotic weighting function 
is
essentially governed by a single $\mathrm{sinc}$ function.
The role of $\mathrm{sinc}(n+r+1)$
could be
neglect,
since its argument
would be large enough.
Indeed,
this function
approaches to zero
according to $O(1/n)$.
Thus, in this case
we say that
\begin{align*}
\lim_{N\to\infty}
w_n(r)
\approx
\mathrm{sinc}(n-r)
.
\end{align*}

Fig{.}~\ref{fig.approx}
shows some plots of $\mathrm{sinc}(n-r)$ for several values of $N$.
These plots
intuitively indicate
the use of a linear approximation for the main lobe of
the $\mathrm{sinc}$
function.
Additionally, 
we assumed that the effect of the remaining lobes is negligible.
We also observed
that,
for $[r]\in \{-1, N \}$,
both $\mathrm{sinc}$ functions are relevant
since they overlap each other.
This last situation was treated separately.

\begin{figure}
\centering
\subfigure[]{\epsfig{file=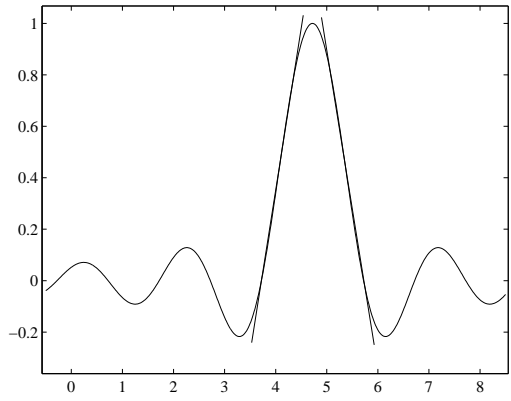,width=8cm}}
\subfigure[]{\epsfig{file=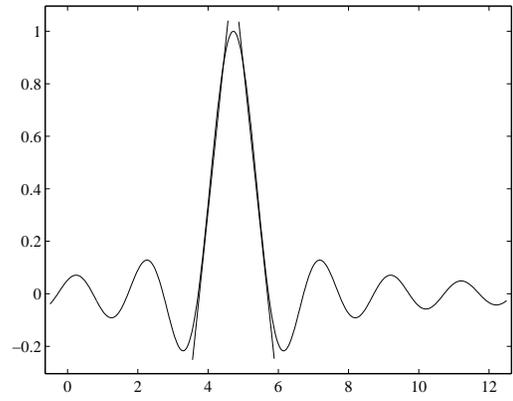,width=8cm}}
\subfigure[]{\epsfig{file=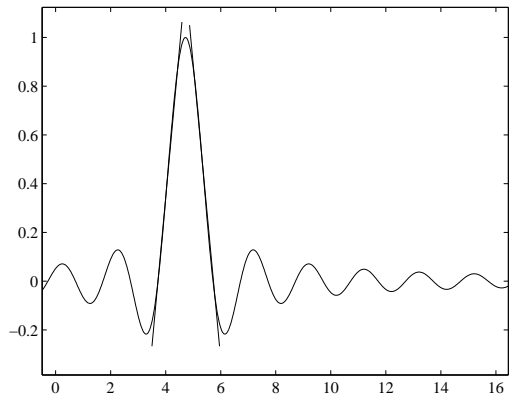,width=8cm}}
\caption{Linear approximation for the main lobe of $\mathrm{sinc}(n-r)$, $r = 4.72$, for $N=8, 12, 16$.}
\label{fig.approx}
\end{figure}

In terms of the above discussion,
we derived an empirical approximation procedure
that considers at most two uniform samples
to render an interpolated sample.
In Fig{.}~\ref{algorithm.heuristic},
algorithmic details of the suggested approximate interpolation are given.
In the proposed heuristic algorithm,
we admit an auxiliary quantity~$\Delta = r - [r]$,
an error tolerance~$\epsilon$,
and
a scaling factor~$\alpha\approx 1.2$, for $N=8$.
This last quantity was found according to 
a standard linear fitting procedure.

\begin{figure}
\hrulefill

\footnotesize
   \textbf{Input:} Fractional number $r$, tolerance $\epsilon$, block-length size $N$. \\
   \textbf{Output:} Approximate values of the ACT weighting function. \\
   \textbf{Method:} Heuristic approximation. \\ [-0.4cm]

\hrulefill

\begin{algorithmic}[]
\STATE $\alpha \leftarrow 1.2$, $\Delta \leftarrow r - [r]$, $\mathbf{w} \leftarrow \mathbf{0}$

\IF{$|\Delta|< \epsilon$}
  \STATE $w_{[r]} = 1$
  \RETURN
\ENDIF

\IF{$[r] \in [1,N-2]$}
  \STATE $w_{[r]-1} \leftarrow (|\Delta|-\Delta)/2$
  \STATE $w_{[r]} \leftarrow 1-|\Delta|$
  \STATE $w_{[r]+1} \leftarrow (|\Delta|+\Delta)/2$
\ELSIF{$[r] = 0$}
  \STATE $w_{[r]} \leftarrow 1-|\Delta|$
  \STATE $w_{[r]+1} \leftarrow (|\Delta|+\Delta)/2$
\ELSIF{$[r] = N-1$}
  \STATE $w_{[r]-1} \leftarrow (|\Delta|-\Delta)/2$
  \STATE $w_{[r]} \leftarrow 1-|\Delta|$
\ELSIF{$[r] = -1$}
  \STATE $w_0 \leftarrow 1$
  \STATE $w_1 \leftarrow -0.35$
\ELSIF{$[r] = N$}
  \STATE $w_{N-2} \leftarrow -0.35$
  \STATE $w_{N-1} \leftarrow 1$
\ENDIF

\RETURN $\alpha \cdot \mathbf{w}$
\end{algorithmic}
\hrulefill
\caption{Algorithm for computing approximate values of the weighting function $w_i(r)$, $i = 0, 1, \ldots, N-1$.}
\label{algorithm.heuristic}
\end{figure}

Considering a conservative choice of $\epsilon = 0.1$,
we employed the proposed heuristics
to calculate the approximate DCT spectra
of 256 
randomly generated signal vectors.
A short block-length $N=8$ was deliberately selected.
The elements of the input signal
were chosen
to be distributed according to
a standard uniform distribution.
The resulting
average
MSE due to the discussed approximation
was as low as $4.7\times10^{-3}$.
This figure is comparable to
the MSE associated
to some integer approximation algorithms for 
the $8$-point  DCT
(e.g., integer cosine transform, $C$-matrix transform)
as detailed in~\cite{britanak2007cosine}.

\section{Matrix-vector representation}
\label{section.matrix}

Further insight on the
nature of the ACT
can be obtained
when the previous constructions are represented
in matrix-vector form.
Let the input signal and its associated DCT spectrum
be denoted by column vectors
$
\mathbf{v}=
\begin{bmatrix}
v_0, & v_1, & \cdots, & v_{N-1}
\end{bmatrix}^T
$
and
$
\mathbf{V}=
\begin{bmatrix}
V_0, & V_1, & \cdots, & V_{N-1}
\end{bmatrix}^T
$
,
respectively.
Additionally,
consider the DCT matrix
$\mathbf{C}$,
whose elements are defined 
according to~(\ref{eq.definition.dct}):
$[\mathbf{C}]_{k,n} = \sqrt{2/N}\alpha_k \cos(\pi k (n+1/2)/N)$,
for $k,n = 0, 1, \ldots, N-1$.
The above quantities are related by $\mathbf{V} = \mathbf{C} \cdot \mathbf{v}$
and $\mathbf{C}^{-1}  = \mathbf{C}^T$~\cite[p.~41]{britanak2007cosine}.

In order
to render the ACT structure in matrix-vector form,
we need to introduce some special matrices.
\begin{definition}[M\"obius matrix]
The $(i,j)$ element of the $N$-order M\"obius matrix $\mathbf{M}_N$
is given by 
$\mu(j/i)$, 
whenever $i$ divides $j$,
for $i,j=1,2,\ldots,N$.
Otherwise, it is zero.
\end{definition}
By construction,
the M\"obius matrix is upper triangular with unity diagonal elements.
Thus, it is always non-singular and its determinant is unity for any dimension.
The inverse of the M\"obius matrix can be directly obtained 
without  calling an inversion procedure.
The $(i,j)$ element of $\mathbf{M}^{-1}$ is 1, if $i$ divides $j$;
otherwise, it is zero.
This fact stems from the M\"obius inversion formula relations.

Additionally,
we consider the extended M\"obius matrix defined by
\begin{align*}
\mathbf{M}'
\triangleq
\mathrm{diag}
(1,\mathbf{M}_{N-1})
=
\begin{bmatrix}
1 & 0 & \cdots & 0 \\
0 &   &        &   \\
\vdots & & \mathbf{M}_{N-1} & \\
0 &   &        &   
\end{bmatrix}
,
\end{align*}
where the operator $\mathrm{diag}(\cdot)$ returns a diagonal matrix. 
We may also admit a vector of averages expressed by
$
\mathbf{S}
=
\begin{bmatrix}
S_0, & S_1, & \cdots, & S_{N-1}
\end{bmatrix}^T
$,
where 
the zeroth average is defined
separately
as
$
S_0
\triangleq
\frac{\sqrt{2}}{N}
\sum_{i=0}^{N-1}
v_i
$,
and
the remaining values
are the $k$th ACT averages
for $k=1,2,\ldots,N-1$.
Notice that DC component of the spectrum 
is related to the zeroth average
$V_0 = \sqrt{N/2}S_0$.

The ACT framework 
can provide a new insight into the DCT spectrum.
In fact,
equation~(\ref{eq.with.Mertens})
indicates that the DCT spectrum
can be separated into two parts:
(i) one due to the M\"obius combination of the ACT averages
and
(ii) another due to the Mertens function.
Let these two parts be termed 
the M\"obius and 
the Mertens parts of the DCT spectrum
denoted by $\mathbf{V_1}$ and $\mathbf{V_2}$,
respectively.
Thus, 
the DCT spectrum can be decomposed as
$\mathbf{V} = \mathbf{V_1} + \mathbf{V_2}$.

Joining the above structures
it follows that
the vector $\mathbf{S}$
is related to the M\"obius part of the DCT spectrum
via the M\"obius extended matrix
according to
\begin{align}
\label{eq.V=M'S}
\mathbf{V_1}
=
\sqrt{\frac{N}{2}}
\cdot
\mathbf{M}'
\cdot
\mathbf{S}
.
\end{align}

In the light of the proposed ACT interpolation,
we can recast the ACT averages 
in terms of the weighting function.
Thus, invoking~(\ref{eq.interpolation}),
we have
\begin{align*}
S_k
&=
\frac{1}{k}
\sum_{m=0}^{k-1}
v_{2m\frac{N}{k}-\frac{1}{2}}
\\
&=
\frac{1}{k}
\sum_{m=0}^{k-1}
\left(
\sum_{n=0}^{N-1}
v_n
w_n\left( 2m\frac{N}{k}-\frac{1}{2} \right)
\right)
,
\quad
k = 1, \ldots, N-1
.
\end{align*}
Inverting the order of the summations,
we may also obtain
\begin{align*}
S_k
&=
\sum_{n=0}^{N-1}
\left(
\frac{1}{k}
\sum_{m=0}^{k-1}
w_n\left( 2m\frac{N}{k}-\frac{1}{2} \right)
\right)
v_n
,
\quad
k = 1, \ldots, N-1
.
\end{align*}
For a fixed $n$,
the inner expression in parenthesis is 
an average of 
particular weighting values at different interpolating points.
This term depends only on $n$ and $k$
being independent of the input vector $\mathbf{v}$.
Then, 
we can separate this expression for better understanding.
Let us define the $k$th weighting average as
\begin{align*}
W_{k,n}
\triangleq
\frac{1}{k}
\sum_{m=0}^{k-1}
w_n\left( 2m\frac{N}{k}-\frac{1}{2} \right),
\quad
k = 1,\ldots, N-1,\ 
n = 0,\ldots, N-1.
\end{align*}
These averages give rise to
the construction of the following matrix:
\begin{align*}
\mathbf{W}
=
\left[
W_{k,n}
\right],
\quad
k = 1,\ldots, N-1,
n = 0,\ldots, N-1.
\end{align*}

Since the rows of $\mathbf{W}$ are
convex combinations of the weighting function,
Proposition~\ref{prop.w.sums.1}
indicates that 
the sum of elements across rows of $\mathbf{W}$ is equal to 1,
i.e.,
$\sum_{n=0}^{N-1} W_{k,n} = 1$,
for $k = 1,\ldots, N-1$.
By augmenting the matrix $\mathbf{W}$
with the inclusion of an extra row and padding,
as described below,
\begin{align*}
\mathbf{W}'
= 
\begin{bmatrix}
1/N & 1/N & \cdots & 1/N \\
0 &   &        &   \\
\vdots & & \mathbf{W} & \\
0 &   &        &   
\end{bmatrix}
,
\end{align*}
we may write the following relation:
\begin{align*}
\mathbf{S}
=
\boldsymbol{\alpha}
\cdot
\mathbf{W}'
\cdot
\mathbf{v}
,
\end{align*}
where
$\boldsymbol{\alpha} \triangleq \mathrm{diag}(\alpha_0, \alpha_1, \ldots, \alpha_{N-1})$.

An application of this last expression into~(\ref{eq.V=M'S})
furnishes the final matrix-vector form 
for the M\"obius part of the considered procedure.
Hence,
\begin{align*}
\mathbf{V_1}
=
\sqrt{\frac{N}{2}}
\cdot
\mathbf{M}'
\cdot
\boldsymbol{\alpha}
\cdot
\mathbf{W}'
\cdot
\mathbf{v}
.
\end{align*}
Effectively,
the transformation matrix that relates $\mathbf{V_1}$ and $\mathbf{v}$
is given by
\begin{align*}
\mathbf{C_1}
\triangleq
\sqrt{\frac{N}{2}}
\cdot
\mathbf{M}'
\cdot
\boldsymbol{\alpha}
\cdot
\mathbf{W}'
.
\end{align*}

Now considering the spectral part $\mathbf{V_2}$
due to the Mertens function in~(\ref{eq.with.Mertens}),
we advance the following additional matrix:
\begin{align*}
\mathbf{C_2}
&=
-
\sqrt{\frac{1}{2N}}
\cdot
\mathrm{diag}
\left(
0, 
M\left(\left\lfloor \frac{N-1}{1} \right\rfloor\right),
M\left(\left\lfloor \frac{N-1}{2} \right\rfloor\right),
\ldots,
M\left(\left\lfloor \frac{N-1}{N-1} \right\rfloor\right)
\right)
\cdot
\mathbf{1}_N
\cdot
\mathbf{1}_N^T
,
\end{align*}
where $\mathbf{1}_N$ is a column vector consisting of
unit elements.
Of course,
if the input signal has null mean,
then the Mertens part of the DCT spectrum 
is always zero:
$\mathbf{C_2} \cdot \mathbf{v} = \mathbf{0}$.

In view of the above,
we must have
\begin{align*}
\mathbf{V}
&=
\mathbf{V_1} + \mathbf{V_2}
\\
&=
\mathbf{C_1}\cdot\mathbf{v}
+
\mathbf{C_2}\cdot\mathbf{v}
\\
&=
(\mathbf{C_1}+\mathbf{C_2}) \cdot \mathbf{v}.
\end{align*}
This manipulation suggests
that the DCT matrix~$\mathbf{C}$
can be
decomposed as
$\mathbf{C}=\mathbf{C}_1+\mathbf{C}_2$.
Notice that
when the input signal has null mean,
a new formulation for the DCT matrix arises.
In this case,
we can establish that
\begin{align*}
\mathbf{C}
\cdot
\mathbf{v}
=
\mathbf{C}_1
\cdot
\mathbf{v}
.
\end{align*}
This alternative transformation matrix for the DCT 
can be considered
as a starting point to derive new algorithms
under the constraint of null mean input signals.

\section{Conclusion and Remarks}
\label{section.conclusions}

The main contributions of this paper
are 
(i) a new arithmetic transform
and
(ii) the elucidation of the arithmetic transform interpolation issues.
The introduced arithmetic cosine transform is a number-theoretic based algorithm
devoted to the DCT computation.
Therefore,
DSP applications that require a DCT evaluation
are potential candidates for the use of the ACT method.

Differently from the standard AFT analysis,
we generalized existing inversion formula,
allowing new frameworks for the arithmetic transform theory.
Besides the M\"obius sequence,
we identified alternative
adequate sequences for
the suggested method.

Moreover, 
the introduced interpolation method
allows the exact computation of the DCT spectrum,
even for small block-lengths.
This is particularly distinct from the existing AFT algorithms,
which 
(i) inevitably introduce approximation errors and 
(ii) tend to excel only when large block-lengths are considered.
Using arithmetic methods (e.g., AFT)
for small block-length transform evaluation
was a challenging issue, 
for which area literature could not furnish adequate solutions until now.

The complexity of the ACT procedure is 
mainly due to 
the interpolation step.
If the sampling process could be adjusted 
in order
to collect samples natively in 
an appropriate nonuniform fashion,
then 
the ACT computation would not need 
any sort of interpolation.
Indeed,
the DCT would be exactly
computed after 
some few
additions/subtractions
associated to the M\"obius function.
Other than that,
the arithmetic transform philosophy can concentrate the transform computational complexity
into the interpolation block.
This is a different perspective for the design of transform procedures.
For instance,
fast algorithms for fractional delay filtering~\cite{valimaki2001fractional}
now receive an additional motivation under the arithmetic cosine transform paradigm.

Finally,
the existence of efficient frameworks
for bidimensional AFTs
suggests a venue for extending
the proposed ACT into the bidimensional case~\cite{ge1997efficient,atlas1997bruns,tufts1988symbiosis,tan2007cmos}.
Without much effort,
the proposed method could be used as the fundamental building block for
a bidimensional arithmetic cosine transform.
On the other hand,
a dedicate analysis for the bidimensional case could prove to be more appropriate.
In particular,
an application Feig-Winograd direct approach could avoid row-column methods~\cite{feig1992fast}.
This could be a more attractive method for the bidimensional ACT.
In any case,
the characterization of the exact interpolation for bidimensional signals
is an open topic.

\section*{Acknowledgment}
\addcontentsline{toc}{section}{Acknowledgment}

The first author wishes to express his gratitude to
Prof.~H\'elio~M.~de Oliveira who provided several
stimulating scientific discussions.
This work was partially supported by the
Department of Foreign Affairs, Trade and Development, Canada.

{\small
\bibliographystyle{IEEEtran}
\bibliography{act}
}

\appendix

\section{The Dirichlet inverse of $\{ (-1)^n \}$}
\label{appendix.inverse.alternating}

In order to derive the Dirichlet inverse of the sequence
$a_n = (-1)^n$, for $n=1,2,3,\ldots$
,
let us examine its associated Dirichlet series.
A result from the theory of functions~\cite[p.~337]{hardy2008numbers}
states that
\begin{align*}
\sum_{n=1}^\infty
\frac{(-1)^n}{n^s}
=
-
(1-2^{1-s}) \zeta(s)
,
\quad
\Re(s)>0,
\end{align*}
where 
$\zeta(s) \triangleq \sum_{n=1}^\infty 1/n^s$
is the Riemann zeta function.
Therefore, we directly have that the closed form of Dirichlet series of $\{a_n\}$,
$A(s)$, is
\begin{align*}
A(s)
=
-
(1-2^{1-s})
\zeta(s)
.
\end{align*}
The Dirichlet inverse of $\{a_n\}$ is a sequence $\{b_n\}$
such that its Dirichlet series, $B(s)=\sum_{n=1}^\infty b_n/n^s$, is equal to $1/A(s)$.
Thus, 
we can maintain that
\begin{align*}
B(s)
&
=
-
\frac{1}{1-2^{1-s}}
\frac{1}{\zeta(s)}
=
-
\frac{1}{1-2^{1-s}}
\sum_{n=1}^\infty
\frac{\mu(n)}{n^s}
.
\end{align*}
Before finding $\{b_n\}$,
we must identify the Dirichlet series of
$1/(1-2^{1-s})$.
This is necessary
in order to 
put $B(s)$ as a product of two Dirichlet
series and apply the convolution theorem
for Dirichlet series.
Accordingly, 
we have that
\begin{align*}
\frac{1}{1-2^{1-s}}
&=
\sum_{k=0}^\infty
(2^{1-s})^k
=
\sum_{k=0}^\infty
\frac{2^k}{(2^k)^s}
.
\end{align*}
This final expression is already in the Dirichlet series format.
Therefore,
the sequence associated to $1/(1-2^{1-s})$
is simply
\begin{align*}
c_n =
\begin{cases}
n, & \text{if $n$ is a power of two,} \\
0, & \text{otherwise.}
\end{cases}
\end{align*}
Returning to the expression for $B(s)$,
we can write as
\begin{align*}
B(s)
&=
-
\sum_{n=1}^\infty
\frac{c_n}{n^s}
\sum_{n=1}^\infty
\frac{\mu(n)}{n^s}
=
-
\sum_{n=1}^\infty
\frac{(c \circledast \mu)(n)}{n^s},
\end{align*}
where $\circledast$ denotes the Dirichlet convolution.
By the equivalence property of Dirichlet series,
we conclude that $b_n = -(c \circledast \mu)(n)$.

Now let us evaluate $(c \circledast \mu)(n)$.
Observe that if $n$ is odd,
then the only divisor of $n$ which is a power of two is
the unit.
Therefore,
we obtain
\begin{align*}
(c \circledast \mu)(n)
&=
\sum_{d|n} c_d \mu(n/d) 
=
c_1 \mu(n) 
=
\mu(n)
.
\end{align*}
On the other hand,
admit that $n$ is even in the form $n = 2^m s$,
where $s$ is an odd integer and $m$ is a positive integer.
Considering that
(i) 
the Dirichlet convolution of two multiplicative functions
results in a multiplicative function~\cite[p.~35]{apostol1984analytic}
and
(ii)
the sequence $\{c_n\}$ is multiplicative with respect to $n$
(a fairly direct result),
we maintain that
\begin{align*}
(c \circledast \mu)(n)
&=
(c \circledast \mu)(2^m)
\cdot
(c \circledast \mu)(s)
=
(c \circledast \mu)(2^m)
\cdot
\mu(s)
.
\end{align*}
According to the definition of the Dirichlet convolution,
the expansion of
$(c \circledast \mu)(2^m)$
yields
\begin{align*}
(c \circledast \mu)(2^m)
&=
c_1 \mu(2^m)
+
c_2 \mu(2^{m-1})
+
\cdots
+
c_{2^{m-1}} \mu(2)
+
c_{2^m} \mu(1).
\end{align*}
Due to the M\"obius function,
only the last two terms are possibly nonnull.
Thus,
we have
\begin{align*}
(c \circledast \mu)(2^m)
&=
c_{2^{m-1}} \mu(2)
+
c_{2^m} \mu(1)
=
2^{m-1} (-1) + 2^m (1)
=
2^{m-1}
.
\end{align*}
Joining the above manipulations,
we have that
\begin{align*}
(c \circledast \mu)(n)
=
\begin{cases}
\mu(n), & \text{if $n$ is odd,} \\
2^{m-1} \mu(2^{-m} n), & \text{if $n = 2^m s$, where $s$ is odd.}
\end{cases}
\end{align*}

\end{document}

%% file: diag_ACT8.pstex_t
\begin{picture}(0,0)%
\includegraphics{diag_ACT8}%
\end{picture}%
\setlength{\unitlength}{4144sp}%
\begingroup\makeatletter\ifx\SetFigFont\undefined%
\gdef\SetFigFont#1#2#3#4#5{%
  \reset@font\fontsize{#1}{#2pt}%
  \fontfamily{#3}\fontseries{#4}\fontshape{#5}%
  \selectfont}%
\fi\endgroup%
\begin{picture}(3627,1935)(526,-1423)
\put(1711,209){\makebox(0,0)[lb]{\smash{{\SetFigFont{8}{9.6}{\rmdefault}{\mddefault}{\updefault}{\color[rgb]{0,0,0}$v_{\frac{25}{14}}$}%
}}}}
\put(1711,389){\makebox(0,0)[lb]{\smash{{\SetFigFont{8}{9.6}{\rmdefault}{\mddefault}{\updefault}{\color[rgb]{0,0,0}$v_{-\frac{1}{2}}$}%
}}}}
\put(1711, 29){\makebox(0,0)[lb]{\smash{{\SetFigFont{8}{9.6}{\rmdefault}{\mddefault}{\updefault}{\color[rgb]{0,0,0}$v_{\frac{13}{16}}$}%
}}}}
\put(1711,-151){\makebox(0,0)[lb]{\smash{{\SetFigFont{8}{9.6}{\rmdefault}{\mddefault}{\updefault}{\color[rgb]{0,0,0}$v_{\frac{27}{10}}$}%
}}}}
\put(1711,-511){\makebox(0,0)[lb]{\smash{{\SetFigFont{8}{9.6}{\rmdefault}{\mddefault}{\updefault}{\color[rgb]{0,0,0}$v_{\frac{57}{14}}$}%
}}}}
\put(1711,-871){\makebox(0,0)[lb]{\smash{{\SetFigFont{8}{9.6}{\rmdefault}{\mddefault}{\updefault}{\color[rgb]{0,0,0}$v_{\frac{59}{10}}$}%
}}}}
\put(1711,-1051){\makebox(0,0)[lb]{\smash{{\SetFigFont{8}{9.6}{\rmdefault}{\mddefault}{\updefault}{\color[rgb]{0,0,0}$v_{\frac{89}{14}}$}%
}}}}
\put(1711,-1231){\makebox(0,0)[lb]{\smash{{\SetFigFont{8}{9.6}{\rmdefault}{\mddefault}{\updefault}{\color[rgb]{0,0,0}$v_{\frac{15}{2}}$}%
}}}}
\put(1711,-331){\makebox(0,0)[lb]{\smash{{\SetFigFont{8}{9.6}{\rmdefault}{\mddefault}{\updefault}{\color[rgb]{0,0,0}$v_{\frac{7}{2}}$}%
}}}}
\put(1711,-691){\makebox(0,0)[lb]{\smash{{\SetFigFont{8}{9.6}{\rmdefault}{\mddefault}{\updefault}{\color[rgb]{0,0,0}$v_{\frac{29}{6}}$}%
}}}}
\put(4051,344){\makebox(0,0)[lb]{\smash{{\SetFigFont{8}{9.6}{\rmdefault}{\mddefault}{\updefault}{\color[rgb]{0,0,0}$V_1$}%
}}}}
\put(4051, 74){\makebox(0,0)[lb]{\smash{{\SetFigFont{8}{9.6}{\rmdefault}{\mddefault}{\updefault}{\color[rgb]{0,0,0}$V_2$}%
}}}}
\put(4051,-196){\makebox(0,0)[lb]{\smash{{\SetFigFont{8}{9.6}{\rmdefault}{\mddefault}{\updefault}{\color[rgb]{0,0,0}$V_3$}%
}}}}
\put(4051,-466){\makebox(0,0)[lb]{\smash{{\SetFigFont{8}{9.6}{\rmdefault}{\mddefault}{\updefault}{\color[rgb]{0,0,0}$V_4$}%
}}}}
\put(4051,-736){\makebox(0,0)[lb]{\smash{{\SetFigFont{8}{9.6}{\rmdefault}{\mddefault}{\updefault}{\color[rgb]{0,0,0}$V_5$}%
}}}}
\put(4051,-1006){\makebox(0,0)[lb]{\smash{{\SetFigFont{8}{9.6}{\rmdefault}{\mddefault}{\updefault}{\color[rgb]{0,0,0}$V_6$}%
}}}}
\put(4051,-1276){\makebox(0,0)[lb]{\smash{{\SetFigFont{8}{9.6}{\rmdefault}{\mddefault}{\updefault}{\color[rgb]{0,0,0}$V_7$}%
}}}}
\put(2836,344){\makebox(0,0)[lb]{\smash{{\SetFigFont{8}{9.6}{\rmdefault}{\mddefault}{\updefault}{\color[rgb]{0,0,0}$S_1$}%
}}}}
\put(2836, 74){\makebox(0,0)[lb]{\smash{{\SetFigFont{8}{9.6}{\rmdefault}{\mddefault}{\updefault}{\color[rgb]{0,0,0}$S_2$}%
}}}}
\put(2836,-196){\makebox(0,0)[lb]{\smash{{\SetFigFont{8}{9.6}{\rmdefault}{\mddefault}{\updefault}{\color[rgb]{0,0,0}$S_3$}%
}}}}
\put(2836,-466){\makebox(0,0)[lb]{\smash{{\SetFigFont{8}{9.6}{\rmdefault}{\mddefault}{\updefault}{\color[rgb]{0,0,0}$S_4$}%
}}}}
\put(2836,-736){\makebox(0,0)[lb]{\smash{{\SetFigFont{8}{9.6}{\rmdefault}{\mddefault}{\updefault}{\color[rgb]{0,0,0}$S_5$}%
}}}}
\put(2836,-1006){\makebox(0,0)[lb]{\smash{{\SetFigFont{8}{9.6}{\rmdefault}{\mddefault}{\updefault}{\color[rgb]{0,0,0}$S_6$}%
}}}}
\put(2836,-1276){\makebox(0,0)[lb]{\smash{{\SetFigFont{8}{9.6}{\rmdefault}{\mddefault}{\updefault}{\color[rgb]{0,0,0}$S_7$}%
}}}}
\put(541,164){\makebox(0,0)[lb]{\smash{{\SetFigFont{8}{9.6}{\rmdefault}{\mddefault}{\updefault}{\color[rgb]{0,0,0}$v_0$}%
}}}}
\put(541,-16){\makebox(0,0)[lb]{\smash{{\SetFigFont{8}{9.6}{\rmdefault}{\mddefault}{\updefault}{\color[rgb]{0,0,0}$v_1$}%
}}}}
\put(541,-196){\makebox(0,0)[lb]{\smash{{\SetFigFont{8}{9.6}{\rmdefault}{\mddefault}{\updefault}{\color[rgb]{0,0,0}$v_2$}%
}}}}
\put(541,-376){\makebox(0,0)[lb]{\smash{{\SetFigFont{8}{9.6}{\rmdefault}{\mddefault}{\updefault}{\color[rgb]{0,0,0}$v_3$}%
}}}}
\put(541,-556){\makebox(0,0)[lb]{\smash{{\SetFigFont{8}{9.6}{\rmdefault}{\mddefault}{\updefault}{\color[rgb]{0,0,0}$v_4$}%
}}}}
\put(541,-736){\makebox(0,0)[lb]{\smash{{\SetFigFont{8}{9.6}{\rmdefault}{\mddefault}{\updefault}{\color[rgb]{0,0,0}$v_5$}%
}}}}
\put(541,-916){\makebox(0,0)[lb]{\smash{{\SetFigFont{8}{9.6}{\rmdefault}{\mddefault}{\updefault}{\color[rgb]{0,0,0}$v_6$}%
}}}}
\put(541,-1096){\makebox(0,0)[lb]{\smash{{\SetFigFont{8}{9.6}{\rmdefault}{\mddefault}{\updefault}{\color[rgb]{0,0,0}$v_7$}%
}}}}
\end{picture}%